\documentclass[a4paper,11pt]{amsart}          
\usepackage{amsfonts,amsmath,latexsym,amssymb} 
\usepackage{amsthm}      
\usepackage[dvipsnames]{xcolor}          
\usepackage{mathrsfs,upref}         
\usepackage{mathptmx}		    
\usepackage{float}
\usepackage{tikz}
\usepackage{circuitikz}
\usepackage{hyperref}
\usepackage{mathtools}

\newtheorem{theorem}{Theorem}

\theoremstyle{definition}

\newtheorem{definition}[theorem]{Definition}

\newtheorem{claim}[theorem]{Claim}

\DeclareMathOperator\dist{dist}
\renewcommand{\Re}{\operatorname{Re}}
\renewcommand{\Im}{\operatorname{Im}}

\renewcommand{\P}{\mathcal{P}}

\author{Anna Muranova}
 \address{Anna Muranova
5050 Institut f\"ur Diskrete Mathematik,
Steyrergasse 30/II,
8010 Graz, Austria } 

 \email{anna.muranova@gmail.com}

\title[Infinite networks and Dirichlet problem]{The effective impedances of infinite ladder networks and Dirichlet problem on graphs}
\thanks{This research was supported by IRTG 2235 Bielefeld-Seoul ``Searching for the regular in the irregular:
Analysis of singular and random systems".}

\begin{document}

\maketitle



\begin{abstract}
We calculate effective impedances of infinite $LC$- and $CL$- ladder networks as limits of effective impedances of finite network approximations, using a new method, which involves precise mathematical concepts. These concepts are related to the classical differential operators on weighted graphs. As an auxiliary result, we solve a  discrete boundary value Dirichlet problem on a finite ladder network.
\end{abstract}

{\footnotesize
{\bf Keywords:} {weighted graphs, analysis on graphs, electrical network, ladder network, Feynman's ladder,  effective impedance, Laplace operator.} 
\smallskip

{\bf Mathematics Subject Classification 2010:}{ 05C22,  34B45, 39A12,  	94C05,  	78A25.} 
}

\section{Introduction}
In his \emph{Lectures on Physics} Richard Feynman firstly states the effective impedance of an infinite $LC$-network (\cite{Feynman2}). Unfortunately, the result was not carefully explained. The first attempt to explain it is given in \cite{Enk}. Then in \cite{Klimo}, \cite{UA}, \cite{UY} there are more investigations on this topic. Moreover, in \cite{Yoon} the effective impedance is calculated as the limit of effective impedances of finite network approximations, assuming that each edge has an infinitely small resistance. In the same paper the effective impedance of an infinite $CL$-network is calculated. 

In the present paper we consider impedances as functions on $\lambda\in \Bbb C\setminus\{0\}$. Although initially $\lambda=i\omega$ (where $\omega$ is a frequency of alternating current, $i$ is the imaginary unit), we extend a function of impedance from imaginary axis to the whole complex plane (see \cite{Brune}).Then a network is considered as a complex-weighted graph (as in previous papers by the author \cite{Muranova3}, \cite{Muranova1}), whose weights depend on a complex parameter $\lambda$. We consider discrete boundary value  Dirichlet problems (see \cite{Muranova1}) on finite approximations of an infinite ladder. Afterwards, we investigate a convergence of the sequences of effective impedances of  the finite approximations for $LC$- and $CL$- ladders in different domains of the complex plane $\lambda$. We prove that the limits exist everywhere, except some segments on the imaginary axis.

Therefore, we calculate the known effective impedances of ladder networks using a new mathematical method, involving precise mathematical definitions. Moreover, we extend the result for $\lambda=i\omega$, $\omega>0$, to the whole complex plane $\lambda$.
\section{Definition of a network}
Let us remind mathematical definitions of electrical network and effective impedance for finite and infinite networks (see \cite{Muranova3}).

Let $(V,E)$ be a locally finite connected graph, where $V$ is a
set of vertices, $\left\vert V\right\vert \geq 2$, and $E$ is a set of
(unoriented) edges.

Assume that each edge $xy$, where $x,y\in V$, is equipped with a resistance $R_{xy}$,
inductance $L_{xy}$, and capacitance $C_{xy}$, where $R_{xy},L_{xy}\in
\lbrack 0,+\infty )$ and $C_{xy}\in (0,+\infty ]$, which correspond to the
physical resistor, inductor (coil), and capacitor (see e.g. \cite{Feynman1}, \cite{Hughes}
). It will be convenient to use the inverse capacity 
\begin{equation*}
D_{xy}=\frac{1}{C_{xy}}\in \lbrack 0,+\infty ).
\end{equation*}%
We always assume that for any edge $xy\in E$ 
\begin{equation*}
R_{xy}+L_{xy}+D_{xy}>0.
\end{equation*}%
The \emph{impedance} of the edge $xy$ is defined as the following function
of a complex parameter $\lambda $: 
\begin{equation*}
z_{xy}^{(\lambda )}=R_{xy}+L_{xy}\lambda +\frac{D_{xy}}{\lambda }.
\end{equation*}%
Although the impedance has a physical meaning only for $\lambda =i\omega $,
where $\omega $ is the frequency of the alternating current (that is, $%
\omega $ is a positive real number), we will allow $\lambda 
$ to take arbitrary values in $\mathbb{C}\setminus \left\{ 0\right\} $ (cf. 
\cite{Brune}, \cite{Muranova3}, \cite{Muranova1}).

In fact, it will be more convenient to work with the \emph{admittance} $\rho
_{xy}^{(\lambda )}$:

\begin{equation}
\rho _{xy}^{(\lambda )}:=\frac{1}{z_{xy}^{(\lambda )}}=\frac{\lambda }{%
L_{xy}\lambda ^{2}+R_{xy}\lambda +{D_{xy}}}.  \label{rhol}
\end{equation}

Define the \emph{physical Laplacian} $\Delta _{\rho }$ as an operator on
functions $f:V\rightarrow \mathbb{C}$ as follows 
\begin{equation}
\Delta _{\rho }f(x)=\sum_{y\in V:y\sim x}(f(y)-f(x))\rho _{xy}^{(\lambda )},
\label{PhysLaplacian}
\end{equation}%
where $x\sim y$ means that $xy\in E$. For convenience let us extend $\rho
_{xy}^{\left( \lambda \right) }$ to all pairs $x,y\in V$ by setting $\rho
_{xy}^{\left( \lambda \right) }\equiv 0$ if $x\not\sim y.$ Then the summation in (%
\ref{PhysLaplacian}) can be extended to all $y\in V.$

Let us fix a vertex $a_{0}\in V$ and a non-empty subset $B\in V$ such that $%
a_{0}\not\in B$. Set $B_{0}=B\cup \{a_{0}\}$. The physical meaning of $a_{0}$ and $B$ is as follows: at the point $a_{0}
$ we keep a potential $1$, while the set $B$ represents
a ground. We refer to the structure $\Gamma =(V,\rho ,a_{0},B)$ as an \emph{(electrical) network}.

On any finite network we can consider the following discrete boundary value Dirichlet problem
\begin{equation}
\begin{cases}
\Delta _{\rho }v^{(\lambda)}(x)=0\text{\ \ for }x\in V\setminus B_{0}, \\ 
v^{(\lambda)}(x)=0\text{\ \ for }x\in B, \\ 
v^{(\lambda)}(a_{0})=1,%
\end{cases}
\label{Dirpr}
\end{equation}%
where $v^{(\lambda)}:V\rightarrow \Bbb C$ is an unknown function of voltage. This system follows from Ohm's and Kirchhoff's complex laws.
The physical voltage at a vertex $x$ at time $t$ is then $\Re v^{(i\omega)}(x)e^{i\omega t}$ for $\lambda=i\omega$, where $\omega>0$ is a frequency of an alternating voltage (see e.g. \cite{Feynman1}, \cite{Hughes}).

Let us define for any network $\Gamma$ the set 
\begin{equation*}
\Lambda=\{\lambda\in \Bbb C\setminus\{0\}\mbox{ such that } z_{xy}^{(\lambda)}\ne 0\mbox{ for any edge of }\Gamma\}.
\end{equation*}

Further we will consider the Dirichlet problem  (\ref{Dirpr}) on a given network exclusively for $\lambda\in \Lambda$. 

If $v^{(\lambda)}\left( x\right) $ is a solution of (\ref{Dirpr}) then the total current
through $a_{0}$ is equal to 
\begin{equation*}
\sum_{x\in V}(1-v^{(\lambda)}(x))\rho _{xa_{0}}^{(\lambda )},
\end{equation*}%
which motivates the following definition (cf. \cite{Muranova3}).

\begin{definition}
For any $\lambda \in \Lambda $, the \emph{effective admittance} of the
network $\Gamma $ is defined by 
\begin{equation}
\mathcal{P}{(\lambda )}=\sum_{x\in V}(1-v^{(\lambda)}(x))\rho _{xa_{0}}^{(\lambda )},
\label{Pdef}
\end{equation}%
where $v$ is a solution of the Dirichlet problem (\ref{Dirpr}). The \emph{%
effective impedance} of $\Gamma $ is defined by 
\begin{equation*}
Z{(\lambda )}=\frac{1}{\mathcal{P}\left( \lambda \right) }=\frac{1}{%
\sum_{x\in V}(1-v^{(\lambda)}(x))\rho _{xa_{0}}^{(\lambda )}}.
\end{equation*}%
If the Dirichlet problem (\ref{Dirpr}) has no solution for some $\lambda $,
then we set $\mathcal{P}{(\lambda )}=\infty $ and $Z{(\lambda )}=0$.
\end{definition}

Note that $Z{(\lambda )}$ and $\mathcal{P}{(\lambda )}$ take values in $%
\overline{\mathbb{C}}=\mathbb{C}\cup \{\infty \}$. In \cite{Muranova3} it is shown that in the case when (\ref{Dirpr}) has multiple solution, the values of $Z{%
(\lambda )}$ and $\mathcal{P}{(\lambda )}$ are independent of the choice of
the solution $v$. 

The defined effective impedance is a natural generalization  of effective resistance for 
weighted graphs (see e.g. \cite{DS}, \cite{LPW}, \cite{Soardi}).

The approach to an effective impedance for infinite networks, used in the following, is described in \cite{Muranova3}. Moreover, it is a generalization of the classical approach to infinite weighted graphs (see e.g. \cite{Barlow}, \cite{Grimmett}, \cite{Soardi}, \cite{Woess}). The idea is to consider finite network approximations.

Let $\Gamma=(V, \rho, a_0, B)$ be an infinite network.
Let us consider the sequence of finite complex-weighted graphs 
$(V_n, \left.\rho\right|_{V_n})$, where $V_n=\{x\in V\mid\dist (a_0,x)\le n\}$, $n\in\Bbb N$. We denote by
\begin{equation*}
\partial V_n=\{x\in V\mid\dist (a_0,x)= n\}
\end{equation*}
 the \emph{boundary} of the graph $(V_n,\left.\rho\right|_{V_n})$. 
Note that $V_{n+1}= \partial V_{n+1}\cup V_n$. Let us denote $B_n=B\cap V_n$. 

Then 
\begin{equation*}
\Gamma_n=(V_n,\left.\rho\right|_{V_n},a_0,B_n\cup \partial V_n),n\in\Bbb N
\end{equation*}
is a \emph{sequence of finite networks exhausted the infinite network $\Gamma$}.

Let $\P_n(\lambda)$ be the effective impedance of $\Gamma_n$. 
\begin{definition}
Define the \emph{effective admittance} of $\Gamma$ as
\begin{equation*}
\P(\lambda)=\lim_{n\rightarrow\infty}\P_n(\lambda)
\end{equation*}
for those $\lambda \in \Bbb C\setminus \{0\}$ where the limit exists.
\end{definition}
It is reasonable to define the effective impedance of infinite network as
 $Z(\lambda)=1/\P(\lambda)$. It takes values in $\overline {\Bbb C}$.

Let us point out that in this model is assumed that any network has potential zero at the infinity. Therefore an approach to the ladder network here is not completely the same as in \cite{Yoon}, but it will give the same result. In \cite{Yoon} the ladder network is considered more like a fractal (compare to \cite{AlonsoRuiz}, \cite{ChenTeplyaev}).

\section{Finite ladder network}
\label{section3}
In this section we calculate the effective admittance of \emph{any} finite ladder network. The calculation follows the same outline as in \cite{Muranova3}, but we present them here for completeness. 

Consider the finite graph $(V,E)$, where 
\begin{equation*}
V=\{0,1,2,3,4,\dots (2n-3), (2n-2)\}\cup\{2n\}
\end{equation*}%
and $E$ is given by $(2k-2)\sim 2k$, $k=\overline{1,n}$  and $(2k-1)\sim 2k$ for $k=\overline{1,(n-1)}$ . Let us consider a network as on Figure \ref{finiteLadderab}. 

\begin{figure}[H]
\centering
\begin{circuitikz} 
  \draw
 (0,0) node[anchor=south]{$0$}
  (0,0) to[short, l=$\alpha$, o-*] (2,0)
  (2,0) node[anchor=south]{$2$}
  (2,0) to[short,  l=$\alpha$, *-*] (4,0)
  (4,0) node[anchor=south]{$4$}
 (6,0) node[anchor=south]{$2k$}
  (8,0) node[anchor=south]{$(2n-2)$}
  (8,0) to[short,  l=$\alpha$, *-o] (10,0)
  (10,0) node[anchor=south]{$2n$}

  (2,0) to[short,  l=$\beta$, *-o] (2,-2)
  (2,-2) node[anchor=north]{$1$}
  (4,0) to[short,  l=$\beta$, *-o] (4,-2)
  (4,-2) node[anchor=north]{$3$}
  (6,0) to[short,  l=$\beta$, *-o] (6,-2)
  (6,-2) node[anchor=north]{$(2k-1)$}
  (8,0) to[short,  l=$\beta$, *-o] (8,-2)
  (8,-2) node[anchor=north]{$(2n-3)$};

  \draw[dashed] 
  (4,0) to[short, *-*] (6,0) 
  (6,0) to[short, *-] (8,0);
\end{circuitikz}
\caption{Finite ladder network}
\label{finiteLadderab}
\end{figure}
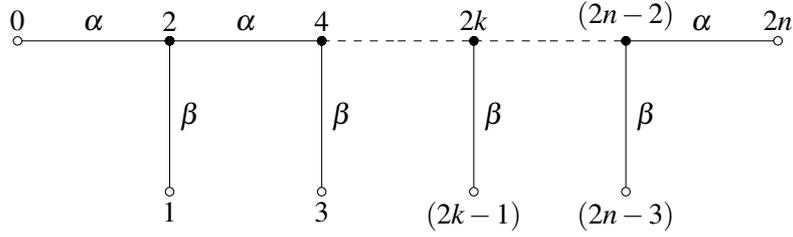

That is, let admittances of the edges $(2k-2)\sim 2k$ be $\alpha=\alpha^{(\lambda)}$ and admittances
of the edges $2k-1\sim 2k$ be $\beta^{(\lambda)}$. Set also $a_0=0$, while 
\begin{equation*}
B=\{1,3,\dots, 2n-3\}\cup\{2n\}.
\end{equation*}%
We will refer to such a network as a \emph{finite $\alpha\beta$-network} and denote it by $\Gamma_n^{\alpha\beta}$. 
The Dirichlet problem (\ref{Dirpr}) for this network is as follows:
\begin{equation}\label{dirprFcomplex}
 \begin{dcases}
 v(2k-2)+\mu v(2k-1)+ v(2k+2)-(2+\mu)v(2k)=0,\ \ k=\overline{1,n-1},
\\
v(0)=1,
\\
v(2k-1)= 0,\ \ k=\overline{1,n-1}, 
\\ 
v(2n)=0,
 \end{dcases}
\end{equation}
where $\mu=\dfrac{\beta}{\alpha}$.

Substituting the equations from the third line of (\ref{dirprFcomplex}) to the first line and denoting $v_k=v(2k)$, we obtain the following recurrence relation for $v_k$:
\begin{equation}\label{recukcomplex}
v_{k+1}-\left(2+\mu\right)v_k+v_{k-1}=0.
\end{equation}
The characteristic polynomial of (\ref{recukcomplex}) is
\begin{equation}\label{psicomplex}
\psi^2-\left(2+\mu\right)\psi+1=0.
\end{equation}
By the definition of a network $\mu\ne 0$ (i.e. $\mu^{(\lambda)}\ne 0$ for $\lambda\in \Lambda$). If $\mu\ne -4$, then the equation (\ref{psicomplex}) has two different complex roots $\psi_1, \psi_2$ and its solution is
\begin{equation}\label{ukeqcomplex}
v_k=c_1\psi_1^k+c_2\psi_2^k,
\end{equation}
where $c_1,c_2\in \mathbb{C}$ are arbitrary constants.

We use the second and fourth equations of (\ref{dirprFcomplex}) as boundary conditions for this recurrence equation.
Substituting (\ref{ukeqcomplex}) in the boundary conditions we obtain the following equations for the constants:
\begin{equation*}
\begin{dcases}
c_1+c_2=1,\\
c_1\psi_1^{n}+c_2\psi_2^{n}=0.
\end{dcases}
\end{equation*}
Therefore,
\begin{equation*}
\begin{dcases}
c_1=\dfrac{1}{1-\psi_1^{2n}}=\dfrac{-\psi_2^{2n}}{1-\psi_2^{2n}},\\
c_2=\dfrac{1}{1-\psi_2^{2n}}=\dfrac{-\psi_1^{2n}}{1-\psi_1^{2n}},
\end{dcases}
\end{equation*}
since $\psi_1\psi_2=1$ by (\ref{psicomplex}).

Now we can calculate the effective admittance of $\Gamma_n^{\alpha\beta}$ in the case $\mu\ne -4$:
\begin{equation}\label{PabNetworkNeminus4complex}
\begin{split}
\mathcal{P}^{\alpha\beta}_{n}=&{\alpha\left(1-v(2)\right)}={\alpha\left(1-v_1\right)}=\alpha\left(1-c_1 \psi_1-c_2 \psi_2\right)\\
=&\dfrac{\alpha\left(\psi_1^{2n-1}+1\right)\left(\psi_1-1\right)}{\left(\psi_1^{2n}-1\right)}=\dfrac{\alpha\left(\psi_2^{2n-1}+1\right)\left(\psi_2-1\right)}{\left(\psi_2^{2n}-1\right)}.
\end{split}
\end{equation}

By definition, an effective admittance $\mathcal P_n^{\alpha\beta}$ is a rational function of $\alpha$ and $\beta$. Indeed, using binomial expansion, it can be written as a rational function of $\alpha$ and $\beta$, without usage of $\psi_1$, $\psi_2$:
\begin{align*}
\mathcal P_n^{\alpha\beta}=&\alpha\left(1-c_1 \psi_1-c_2 \psi_2\right)=\alpha\left(1-\dfrac{\psi_1}{1-\psi_1^{2n}}-\dfrac{\psi_2}{1-\psi_2^{2n}}\right)\\
=&\alpha\left(1-\dfrac{\psi_1\left(1-\psi_2^{2n}\right)+\psi_2\left(1-\psi_1^{2n}\right)}{\left(1-\psi_1^{2n}\right)\left(1-\psi_2^{2n}\right)}\right)\\
=&\alpha\left(1-\dfrac{\psi_1+\psi_2-\left(\psi_2^{2n-1}+\psi_1^{2n-1}\right)}{2-\left(\psi_1^{2n}+\psi_2^{2n}\right)}\right)\\
=&\alpha\left(1-\dfrac{2+\dfrac{\beta}{\alpha}-2\sum\limits_{k=0}^{n-1}{{2n-1}\choose{2k}}\left(1+\dfrac{\beta}{2\alpha}\right)^{2n-2k-1}\left(\dfrac{\beta}{\alpha}+\left(\dfrac{\beta}{2\alpha}\right)^2\right)^{k}}{2-2\sum\limits_{k=0}^{n}{{2n}\choose{2k}}\left(1+\dfrac{\beta}{2\alpha}\right)^{2n-2k}\left(\dfrac{\beta}{\alpha}+\left(\dfrac{\beta}{2\alpha}\right)^2\right)^{k}}\right),
\end{align*}
since $\psi_1+\psi_2=2+\mu=2+\dfrac{\beta}{\alpha}$ by \eqref{psicomplex}.

Let us now consider the case $\mu =-4$. Then the solution of the recurrence relation (\ref{recukcomplex}) is 
\begin{equation*}
v_k=c_1(-1)^k+c_2k(-1)^k,
\end{equation*}
where $c_1,c_2\in \mathbb{C}$ are arbitrary constants. And using boundary conditions for the recurrence relation, we obtain
\begin{equation*}
\begin{dcases}
c_1=1\\
c_2=-\dfrac{1}{n}.
\end{dcases}
\end{equation*}

Then the effective admittance is
\begin{equation*}
\mathcal{P}_{n}={\alpha\left(1-v(2)\right)}={\alpha\left(1-v_1\right)}={\alpha\left(1-\left(-1+\dfrac{1}{n}\right)\right)}=\dfrac{\alpha(2n-1)}{ n}.
\end{equation*}

Therefore, for a finite $\alpha\beta$-network we have for any $\lambda\in\Lambda$
\begin{equation}\label{PabNetworkcomplex}
\mathcal{P}_{n}=
\begin{dcases}
\dfrac{\alpha(2n-1)}{ n}\mbox{, if }\dfrac{\beta}{\alpha}= -4,\\
\P_n^{\alpha\beta}, \mbox{otherwise}.
\end{dcases}
\end{equation}

\section{Feynman's ladder ($LC$-network) with zero at infinity}
\label{section4}
In this section we calculate the effective admittance of an infinite $LC-$network, using its finite approximations and their admittances, calculated in the previous section.
Consider the infinite ladder network $(V,E)$, where 
\begin{equation*}
V=\{0,1,2,3,4,\dots \}
\end{equation*}%
and $E$ is given by $(2k-2)\sim 2k$ and $(2k-1)\sim 2k$ for $k=\overline{%
1,\infty }$  (see Figure \ref{figFLadder})

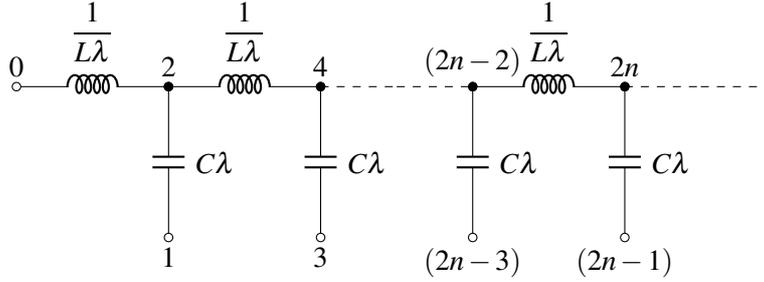
\begin{figure}[H]
\centering
\begin{circuitikz} 
  \draw
 (0,0) node[anchor=south]{$0$}
  (0,0) to[/tikz/circuitikz/bipoles/length=30pt,  L, o-*] (2,0)
  (1,0.2) node[anchor=south]{$\dfrac{1}{L\lambda}$}

  (2,0) node[anchor=south]{$2$}
  (2,0) to[/tikz/circuitikz/bipoles/length=30pt, L, *-*] (4,0)
  (3,0.2) node[anchor=south]{$\dfrac{1}{L\lambda}$}
  (4,0) node[anchor=south]{$4$}

  (2,0) to[/tikz/circuitikz/bipoles/length=20pt,C,  l=$C\lambda$,-o] (2,-2)
  (2,-2) node[anchor=north]{$1$}

  (4,0) to[/tikz/circuitikz/bipoles/length=20pt,C,  l=$C\lambda$,-o] (4,-2)
  (4,-2) node[anchor=north]{$3$}

  (6,0) node[anchor=south]{$(2n-2)$}
  (6,0) to[/tikz/circuitikz/bipoles/length=30pt, L, *-*] (8,0)
  (7,0.2) node[anchor=south]{$\dfrac{1}{L\lambda}$}
  (8,0) node[anchor=south]{$2n$}

  (6,0) to[/tikz/circuitikz/bipoles/length=20pt,C,  l=$C\lambda$,-o] (6,-2)
  (6,-2) node[anchor=north]{$(2n-3)$}
  (8,0) to[/tikz/circuitikz/bipoles/length=20pt,C,  l=$C\lambda$,-o] (8,-2)
  (8,-2) node[anchor=north]{$(2n-1)$};

  \draw[dashed] 
  (4,0) to[short, *-*] (6,0) 
  (8,0) to[short, *-] (10,0);
\end{circuitikz}
\caption{Feynman's ladder with zero at infinity}
\label{figFLadder}
\end{figure}

Let the admittance of the edges $(2k-2)\sim 2k$ be $\dfrac{1}{L \lambda} $ and admittance
of the edges $2k-1\sim 2k$ be $C \lambda$, where $L>0$
and $C>0$. Set also $a_{0}=0$, while $B=\{1,3,\dots \}$.

This network is the Feynman's ladder network (see \cite{Feynman2}) and we assume that it has a ground at infinity. 

In this section we analyse the behaviour of the sequence of the effective admittances for the exhausted networks in the whole complex plane $\lambda$, and compare our calculations with the result stated by Richard Feynman in \cite{Feynman2} and Sung Hyun Yoon in \cite{Yoon} for $\lambda=i\omega$, $\omega>0$.

This network can be exhausted by finite $\alpha\beta$-networks (with $\alpha=\dfrac{1}{L \lambda}$, $\beta={C \lambda}$), whose effective admittances by \eqref{PabNetworkNeminus4complex} and \eqref{PabNetworkcomplex} are
\begin{equation*}
\P_n(\lambda)=
\begin{dcases}
\dfrac{2n-1}{L \lambda n},\mbox{ if } \mu= -4,\\
\mbox{not defined, if } \lambda=0,\\
\P_n^{\alpha\beta}(\lambda)=\dfrac{\left(\psi_1^{2n-1}+1\right)\left(\psi_1-1\right)}{L \lambda\left(\psi_1^{2n}-1\right)}, \mbox{otherwise},
\end{dcases}
\end{equation*}
where $\mu=LC\lambda^2$ and $\psi_1$ is any root of the equation
\begin{equation}\label{quadraticEqFeynman}
\psi^2-\left(2+{LC\lambda^2}\right)\psi+1=0.
\end{equation}

Let us analyse the sequence $\{\P_n^{\alpha\beta}(\lambda)\}$, $n\rightarrow\infty$.

Since $\psi_1$ is any root of the quadratic equation \eqref{quadraticEqFeynman}, we can assume, without loss of generality, that $|\psi_1|\le1\le|\psi_2|$. Note that  then $\{\P_n^{\alpha\beta}(\lambda)\}$ has limit if and only if $|\psi_1|<1$. This limit is equal to 
$(1-\psi_1)/(L \lambda)$. Now we will reformulate the condition on existence of the limit in terms of $\lambda$. Firstly, we will prove the following claim.

\begin{claim} \label{modulus_psi_1}
Let $\lambda\in\Bbb C\setminus\{0\}$, $\lambda^2\ne -\dfrac{4}{LC}$. Then the condition $|\psi_{1}|=|\psi_{2}|=1$ occurs if and only if $\lambda^2 \in \left(-\dfrac{4}{LC},0\right)$.
\end{claim}

\begin{proof}
\textquotedblleft $\Rightarrow $\textquotedblright\ Let $|\psi_{1}|=1$. Then $|\psi_{2}|=1$, since $\psi_1\psi_2=1$ (see (\ref{quadraticEqFeynman})). Further, since $LC \lambda^2 \neq 0$
and $LC \lambda^2 \neq -4$, it follows from (\ref{quadraticEqFeynman}) and $|\psi_{1}|=|\psi_{2}|=1$ that $\psi_{1},\psi_{2}\not\in \mathbb{R}$ (since $\psi_{1},\psi_{2}\ne\pm 1$). And, using $\psi_{1}\overline{\psi_{1}}=|\psi_1|^2=1$, we have $\psi_{2}=\overline{\psi_{1}}$. Moreover, by (\ref{quadraticEqFeynman}) we conclude that $2+LC \lambda^2 =\psi_{1}+\psi_{2}=2\Re \psi_1\in \mathbb{R}$, i.e. $LC \lambda^2 \in \Bbb R$. Then $
\psi_{1},\psi_{2}\not\in \mathbb{R}$ means that the determinant of (\ref{quadraticEqFeynman}) 
\begin{equation*}
\left( 2+LC \lambda^2 \right) ^{2}-4=4LC \lambda^2+(LC \lambda^2) ^{2}
\end{equation*}
is negative, i.e. $LC \lambda^2 \in (-4, 0)$.
Therefore, 
\begin{equation*}
\lambda^2 \in \left(-\dfrac{4}{LC},0\right)
\end{equation*}
which was to be proved.

\textquotedblleft $\Leftarrow $\textquotedblright\ Let $\lambda^2 \in \left(-\dfrac{4}{LC},0\right)$.
Then the determinant of (\ref{quadraticEqFeynman}) is negative and 
\begin{align*}
|\psi_{1,2}|^{2}=&\left\vert 1+\dfrac{LC \lambda^2 }{2}\pm i\sqrt{-LC \lambda^2 -\left( \dfrac{
LC \lambda^2 }{2}\right) ^{2}}\right\vert ^{2}\\
=&\left( 1+\dfrac{LC \lambda^2 }{2}\right)
^{2}-LC \lambda^2 -\left( \dfrac{LC \lambda^2 }{2}\right) ^{2}=1.
\end{align*}
\end{proof}

Therefore, for any $\lambda\in\Bbb C\setminus\left[-i\sqrt{\dfrac{4}{LC}},i\sqrt{\dfrac{4}{LC}}\right]$ we can write 
\begin{equation}\label{psixi}
\psi_1(\lambda)=1+\dfrac{LC\lambda^2}{2}+\lambda\xi(\lambda),
\end{equation}
where $\xi(\lambda)$ is the square root of $LC+\dfrac{L^2C^2\lambda^2}{4}$, such that $|\psi_1(\lambda)|< 1$. 

Let us denote $\gamma=LC+\dfrac{L^2C^2\lambda^2}{4}$. Let us consider $\gamma\in \Bbb C\setminus(-i\infty,0)$, i.e. $\gamma=re^{i\phi}$, $\phi\in \left(-\dfrac{\pi}{2},\dfrac{3\pi}{2}\right)$ in polar coordinates. Then there are exactly two continuous functions $\xi_1, \xi_2$, which give square roots of $\gamma$: 
\begin{equation*}
\xi_1(\gamma)=\sqrt {r} e^{i\dfrac{\phi}{2}},\xi_2(\gamma)=-\xi_1(\gamma)
\end{equation*}
(see Figure \ref{figxi1xi2}).

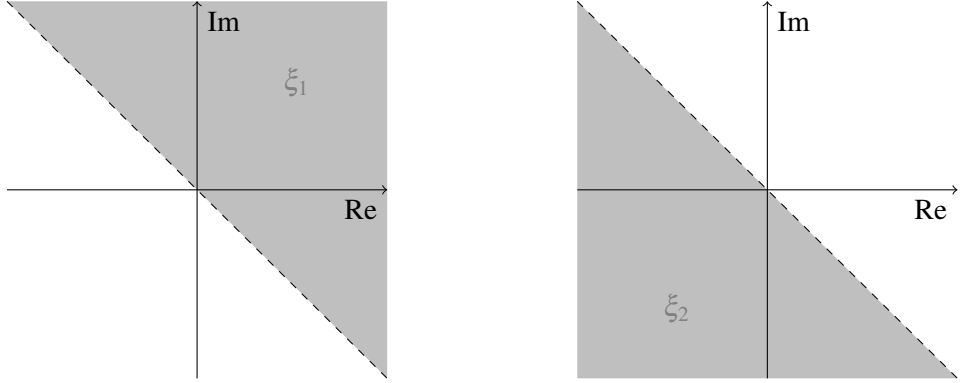
\begin{figure}[H]
\centering
\begin{tikzpicture}[scale=0.5]

   \path [draw=none,fill=gray,semitransparent] (-5,5) -- (5,5) -- (5,-5);

    \draw [->] (-5,0) -- (5,0) node [below left]  {$\Re $};
    \draw [->] (0,-5) -- (0,5) node [below right] {$\Im $};
    \draw [dashed] (-5,5) -- (5,-5);

    \node [below right,gray] at (+2,+3.5) {$\xi_1$};

   \path [draw=none,fill=gray,semitransparent] (10,5) -- (10,-5) -- (20,-5);

    \draw [->] (10,0) -- (20,0) node [below left]  {$\Re$};
    \draw [->] (15,-5) -- (15,5) node [below right] {$\Im$};
    \draw [dashed] (10,5) -- (20,-5);

    \node [below right,gray] at (12,-2.5) {$\xi_2$};
\end{tikzpicture}
\caption{The images of $\xi_1$ and $\xi_2$}
\label{figxi1xi2}
\end{figure}
 
Since 
\begin{equation*}
\begin{split}
\gamma=&LC+\dfrac{L^2C^2\lambda^2}{4}=LC+\dfrac{L^2C^2\left(\Re \lambda+i \Im \lambda\right)^2}{4}\\
=&LC+\dfrac{L^2C^2\left(\Re \lambda\right)^2-L^2C^2 \left(\Im \lambda\right)^2}{4}+i\dfrac{L^2C^2(\Re\lambda)(\Im\lambda)}{2},
\end{split}
\end{equation*}

$\xi_1(\gamma(\lambda))$ and $\xi_2(\gamma(\lambda))$ are defined for all $\lambda\in \Bbb C\setminus \overline{\Lambda}$, where
\begin{equation*}
\overline{\Lambda}=\left\{1+\dfrac{LC\left(\Re \lambda\right)^2-LC \left(\Im \lambda\right)^2}{4}= 0 \mbox{ and } \left(\Re\lambda\right)\left(\Im\lambda\right)<0\right\},
\end{equation*}
see Figure \ref{figxilambdanotdefined}.

\begin{figure}[H]
\centering
\begin{tikzpicture}[scale=0.7]

   \path [draw=none,fill=gray,semitransparent] (-5,-5) rectangle (5,5) ;

    \draw [->] (-5,0) -- (5,0) node [below left]  {$\Re \lambda$};
    \draw [->] (0,-5) -- (0,5) node [below right] {$\Im \lambda$};
   \draw[white, line width = 0.60mm]  plot[smooth,domain=-4.85:0] (\x, {sqrt((\x)^2+4)}); 
    \draw[white, line width = 0.60mm]  plot[smooth,domain=0:4.8] (\x, {-sqrt((\x)^2+4)});
    \draw[dashed, line width = 0.30mm]  plot[smooth,domain=-4.6:0] (\x, {sqrt((\x)^2+4)});
    \draw[dashed, line width = 0.30mm]  plot[smooth,domain=0:4.55] (\x, {-sqrt((\x)^2+4)});

   \node [right,black] at (0,+2) {$i\sqrt{\dfrac{4}{LC}}$};
   \node [left,black] at (0,-2) {$-i\sqrt{\dfrac{4}{LC}}$};

\end{tikzpicture}
\caption{The domain of $\xi_{1,2}(\gamma(\lambda))$}
\label{figxilambdanotdefined}
\end{figure}
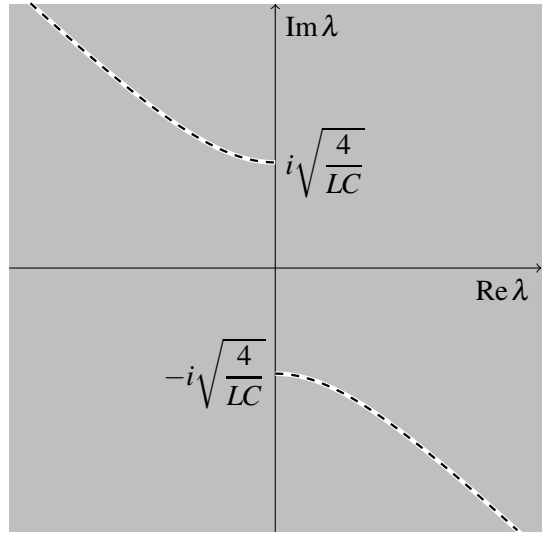

Since the functions
\begin{equation*}
\left|1+\dfrac{LC\lambda^2}{2}+\lambda\xi_1(\gamma(\lambda))\right| \mbox { and }\left|1+\dfrac{LC\lambda^2}{2}+\lambda\xi_2(\gamma(\lambda))\right|
\end{equation*}
are continuous on $\lambda$, the choice of the function $\xi_1(\lambda)$ or $\xi_2(\lambda)$ in \eqref{psixi}, can not change inside the domains 
\begin{equation*}
\Omega_1=\left\{\Re^2 \lambda-\Im^2 \lambda>-\sqrt{\dfrac{4}{LC}}, \Re \lambda>0 \right\}\cup\left\{\Re^2 \lambda-\Im^2 \lambda<-\sqrt{\dfrac{4}{LC}}, \Im \lambda>0 \right\}
\end{equation*}
and
\begin{equation*}
\Omega_2=\left\{\Re^2 \lambda-\Im^2 \lambda>-\sqrt{\dfrac{4}{LC}}, \Re \lambda<0 \right\}\cup\left\{\Re^2 \lambda-\Im^2 \lambda<-\sqrt{\dfrac{4}{LC}}, \Im \lambda<0 \right\}
\end{equation*}
(see Figure \ref{figxilambdanotdefined2}).

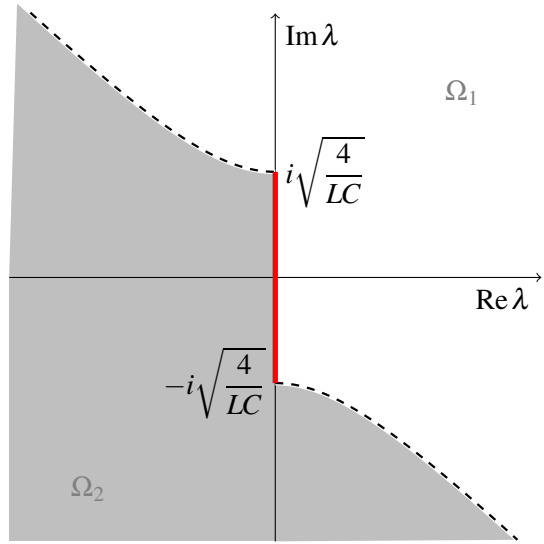
\begin{figure}[H]
\centering
\begin{tikzpicture}[scale=0.7]

   \path [draw=none,fill=gray,semitransparent]
   (-5,0) -- plot[ smooth,domain=-4.85:0] (\x, {sqrt((\x)^2+4)})-- (0,0) -- cycle ;
    \path [draw=none,fill=gray,semitransparent] (-5,-5) rectangle (0,0) ;

   \path [draw=none,fill=gray,semitransparent]
  (0,-5) -- plot[ smooth,domain=0:4.55] (\x, {-sqrt((\x)^2+4)}) -- cycle ;

    \draw [->] (-5,0) -- (5,0) node [below left]  {$\Re \lambda$};
    \draw [->] (0,-5) -- (0,5) node [below right] {$\Im \lambda$};
   \draw[white, line width = 0.60mm]  plot[ smooth,domain=-4.9:0] (\x, {sqrt((\x)^2+4)}); 
    \draw[white, line width = 0.60mm]  plot[smooth,domain=0:4.8] (\x, {-sqrt((\x)^2+4)});
    \draw[dashed, line width = 0.30mm]  plot[smooth,domain=-4.6:0] (\x, {sqrt((\x)^2+4)});
    \draw[dashed, line width = 0.30mm]  plot[smooth,domain=0:4.55] (\x, {-sqrt((\x)^2+4)});
    \draw[red, line width = 0.70mm]  (0,-2) -- (0,2); 

   \node [right,black] at (0,+2) {$i\sqrt{\dfrac{4}{LC}}$};
   \node [left,black] at (0,-2) {$-i\sqrt{\dfrac{4}{LC}}$};
   \node [right,gray] at (3,3.5) {$\Omega_1$};
   \node [left,gray] at (-3,-4) {$\Omega_2$};

\end{tikzpicture}
\caption{The domains $\Omega_1$ and $\Omega_2$}
\label{figxilambdanotdefined2}
\end{figure}

Taking $\lambda=2\in \Omega_1$ we have
\begin{equation*}
\left|1+\dfrac{LC\lambda^2}{2}+\lambda\xi_1(\gamma(\lambda))\right|=\left|1+2 LC+2\xi_1({LC+L^2C^2})\right|
=\left|1+2 LC+2\sqrt{LC+L^2C^2}\right|>1.
\end{equation*}
Therefore,
\begin{equation*}
\psi_1(\lambda)=1+\dfrac{LC\lambda^2}{2}+\lambda\xi_2(\gamma(\lambda)), \lambda\in\Omega_1.
\end{equation*}

In the same way, taking $\lambda=-2\in \Omega_2$ we obtain
\begin{equation*}
\left|1+\dfrac{LC\lambda^2}{2}+\lambda\xi_1(\gamma(\lambda))\right|=\left|1+2 LC-2\xi_1({LC+L^2C^2})\right|
=\left|1+2 LC-2\sqrt{LC+L^2C^2}\right|<1.
\end{equation*}
Therefore,
\begin{equation*}
\psi_1(\lambda)=1+\dfrac{LC\lambda^2}{2}+\lambda\xi_1(\gamma(\lambda)), \lambda\in\Omega_2.
\end{equation*}

Therefore, we can calculate the effective admittance of the infinite network for any  $\lambda \in \Bbb C\setminus\left(\overline\Lambda\cup\left[-i\sqrt{\dfrac{4}{LC}},i\sqrt{\dfrac{4}{LC}}\right]\right)$ as
\begin{equation*}
\P(\lambda)=\dfrac{1-\psi_1}{L\lambda}=
\begin{dcases}
-\dfrac{C\lambda}{2}-\dfrac{\xi_2(\gamma(\lambda))}{L}, \lambda\in\Omega_1,\\
-\dfrac{C\lambda}{2}-\dfrac{\xi_1(\gamma(\lambda))}{L}, \lambda\in\Omega_2.\\
\end{dcases}
\end{equation*}

The effective admittance for the $\lambda\in\overline \Lambda$ we can calculate, considering another 
 cut of the plane $\gamma$. Moreover, since the cut $(-i\infty,0)$ has been chosen arbitrary, the limits of the effective admittances from the both sides of the curves $\overline \Lambda$ will coincide and give the required quantity. 

Therefore, we can calculate the effective admittance of the infinite network for any  $\lambda \in \Bbb C\setminus\left[-i\sqrt{\dfrac{4}{LC}},i\sqrt{\dfrac{4}{LC}}\right]$.

A natural question is whether the right and left limits of the effective admittance (or of the $\xi_1(\lambda)$, $\xi_2(\lambda)$) at the segment $\left[-i\sqrt{\dfrac{4}{LC}},i\sqrt{\dfrac{4}{LC}}\right]$ coincide (see Figure \ref{figxi1xi2rightleftlimits}).  The answer is negative.

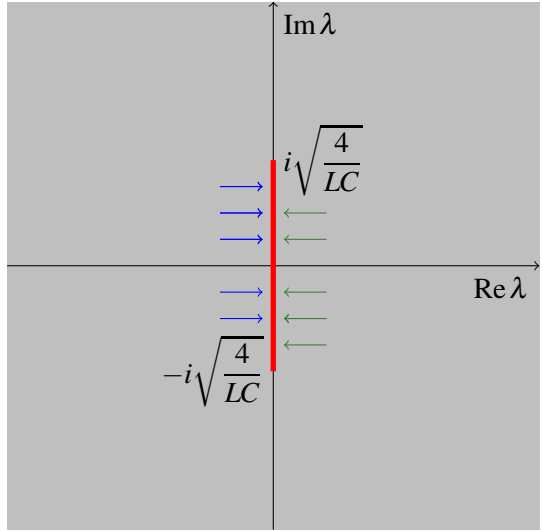
\begin{figure}[H]
\centering
\begin{tikzpicture}[scale=0.7]

   \path [draw=none,fill=gray,semitransparent] (-5,-5) rectangle (5,5);
    \draw [->] (-5,0) -- (5,0) node [below left]  {$\Re \lambda$};
    \draw [->] (0,-5) -- (0,5) node [below right] {$\Im \lambda$};

    \draw[red, line width = 0.70mm]  (0,-2) -- (0,2); 

   \node [right,black] at (0,+2) {$i\sqrt{\dfrac{4}{LC}}$};
   \node [left,black] at (0,-2) {$-i\sqrt{\dfrac{4}{LC}}$};

    \draw [->, blue] (-1,0.5) -- (-0.2,0.5);
    \draw [->, blue] (-1,1) -- (-0.2,1);
    \draw [->, blue] (-1,1.5) -- (-0.2,1.5);
    \draw [->, blue] (-1,-0.5) -- (-0.2,-0.5);
    \draw [->, blue] (-1,-1) -- (-0.2,-1);
    \draw [->, blue] (-1,0.5) -- (-0.2,0.5);
    \draw [->, blue] (-1,1) -- (-0.2,1);
    \draw [->, blue] (-1,1.5) -- (-0.2,1.5);
    \draw [->, OliveGreen] (1,-0.5) -- (0.2,-0.5);
    \draw [->, OliveGreen] (1,-1) -- (0.2,-1);
    \draw [->, OliveGreen] (1,0.5) -- (0.2,0.5);
    \draw [->, OliveGreen] (1,1) -- (0.2,1);
    \draw [->, OliveGreen] (1,-1) -- (0.2,-1);
    \draw [->, OliveGreen] (1,-1.5) -- (0.2,-1.5);
\end{tikzpicture}
\caption{The limits of $\xi_1(\lambda)$ and $\xi_2(\lambda)$}
\label{figxi1xi2rightleftlimits}
\end{figure}

 Indeed, let $\lambda=\epsilon+ i\omega$, $1\gg\epsilon>0$, $\omega\in\left(-\sqrt{\dfrac{4}{LC}},\sqrt{\dfrac{4}{LC}}\right)$, $\lambda\in\Omega_1$. Then
\begin{equation*}
\gamma(\epsilon+i\omega)=LC+\dfrac{L^2C^2}{4}\left(\epsilon^2-\omega^2\right)+ i \dfrac{L^2C^2}{2}\epsilon\omega
\end{equation*}
and 
\begin{equation}\label{limitEffadmF1}
\begin{split}
\lim_{\lambda\rightarrow+i\omega}\P(\lambda)=&\lim_{\epsilon\rightarrow 0}\left(-\dfrac{C(\epsilon+i\omega)}{2}-\dfrac{\xi_2(\gamma(\epsilon+i\omega))}{L}\right)\\
=&  -\dfrac{Ci\omega}{2}+\sqrt{\dfrac{C}{L}-\dfrac{C^2\omega^2}{4}},
\end{split}
\end{equation}
since $\Re \gamma>0$, $|\Im \gamma|\ll 1$ provides $\Re \xi_2<0$. 

Let $\lambda=-\epsilon+ i\omega$, $1\gg\epsilon>0$, $\omega\in\left(-\sqrt{\dfrac{4}{LC}},\sqrt{\dfrac{4}{LC}}\right)$, $\lambda\in\Omega_2$. Then
\begin{equation*}
\gamma(-\epsilon+i\omega)=LC+\dfrac{L^2C^2}{4}\left(\epsilon^2-\omega^2\right)-i \dfrac{L^2C^2}{2}\epsilon\omega
\end{equation*}
and 
\begin{equation*}
\begin{split}
\lim_{\lambda\rightarrow-i\omega}\P(\lambda)=&\lim_{\epsilon\rightarrow 0}\left(-\dfrac{C(-\epsilon+i\omega)}{2}-\dfrac{\xi_1(\gamma(-\epsilon+i\omega))}{L}\right)\\
=&  -\dfrac{Ci\omega}{2}-\sqrt{\dfrac{C}{L}-\dfrac{C^2\omega^2}{4}},
\end{split}
\end{equation*}
since $\Re \gamma>0$, $|\Im \gamma|\ll 1$ provides $\Re \xi_1>0$. 

The limit \eqref{limitEffadmF1} coincides with the one, stated by R. Feynman in \cite{Feynman2} and S.H. Yoon in \cite{Yoon} . Indeed, by \cite[p. 22-13]{Feynman2} we have the effective admittance
\begin{align*}
\P=&\dfrac{1}{Z}=\dfrac{1}{{i\omega L}/{2}+\sqrt{(L/C)-(\omega^2L^2/4)}}=\dfrac{{i\omega L}/{2}-\sqrt{(L/C)-(\omega^2L^2/4)}}{({i\omega L}/{2})^2-((L/C)-(\omega^2L^2/4))}\\
=&-\dfrac{C}{L}\left(\dfrac{i\omega L}{2}-\sqrt{(L/C)-(\omega^2L^2/4)}\right)=-\dfrac{Ci\omega}{2}+\sqrt{\dfrac{C}{L}-\dfrac{C^2\omega^2}{4}}.
\end{align*}

This corresponds to the ideas in \cite{AlonsoRuiz} and in \cite{Yoon} to calculate the effective impedance as right half-plane limit. Physically it makes sense, since the real resistance in any part of a physical network is always greater than zero. Note that we have made a rescaling in a different way than in the above-mentioned papers, although it leads to the same result. Namely, we add a positive real component to $i\omega$, but not to the whole edge. 

\section{$CL$-network with zero at infinity}

In this section we calculate the effective impedance of infinite $CL-$ network, using finite network approximation, described in Section \eqref{section3}.
Consider the infinite ladder network $(V,E)$, where 
\begin{equation*}
V=\{0,1,2,3,4,\dots \}
\end{equation*}%
and $E$ is given by $(2k-2)\sim 2k$ and $(2k-1)\sim 2k$ for $k=\overline{%
1,\infty }$  (see Figure \ref{figCLLadder})

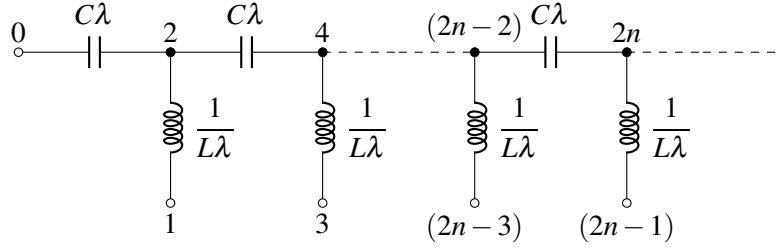
\begin{figure}[H]
\centering
\begin{circuitikz} 
  \draw
 (0,0) node[anchor=south]{$0$}
  (0,0) to[/tikz/circuitikz/bipoles/length=20pt,  C, l=$C\lambda$, o-*] (2,0)

  (2,0) node[anchor=south]{$2$}
  (2,0) to[/tikz/circuitikz/bipoles/length=20pt,  C, l=$C\lambda$, *-*] (4,0)
  (4,0) node[anchor=south]{$4$}

  (2,0) to[/tikz/circuitikz/bipoles/length=30pt,L,  l=$\dfrac{1}{L\lambda}$,-o] (2,-2)
  (2,-2) node[anchor=north]{$1$}

  (4,0) to[/tikz/circuitikz/bipoles/length=30pt,L,  l=$\dfrac{1}{L\lambda}$,-o] (4,-2)
  (4,-2) node[anchor=north]{$3$}

  (6,0) node[anchor=south]{$(2n-2)$}
  (6,0) to[/tikz/circuitikz/bipoles/length=20pt,  C, l=$C\lambda$, *-*] (8,0)
  (8,0) node[anchor=south]{$2n$}

  (6,0) to[/tikz/circuitikz/bipoles/length=30pt,L,  l=$\dfrac{1}{L\lambda}$,-o] (6,-2)
  (6,-2) node[anchor=north]{$(2n-3)$}
  (8,0) to[/tikz/circuitikz/bipoles/length=30pt,L,  l=$\dfrac{1}{L\lambda}$,-o] (8,-2)
  (8,-2) node[anchor=north]{$(2n-1)$};

  \draw[dashed] 
  (4,0) to[short, *-*] (6,0) 
  (8,0) to[short, *-] (10,0);
\end{circuitikz}
\caption{$CL$-ladder with zero at infinity}
\label{figCLLadder}
\end{figure}

Let the admittance of the edges $(2k-2)\sim 2k$ be ${C \lambda} $ and admittance
of the edges $2k-1\sim 2k$ be $\dfrac{1}{L \lambda}$, where $C,L>0$. Set also $a_{0}=0$, while $B=\{1,3,\dots \}$.

This network is a $CL$-ladder network (see \cite{Yoon}). Note that we assume it has 
a ground at infinity. 

We will calculate an effective admittance of the $CL$-ladder in the same way, as we did for Feynman's ladder in Section \ref{section4}, and compare our calculations to the result stated by Sung Hyun Yoon in \cite{Yoon} for $\lambda=i\omega$, $\omega>0$. 

The $CL$-network can be exhausted by finite $\alpha\beta$-networks (with $\alpha=C \lambda$, $\beta=\dfrac{1}{L \lambda}$), whose effective admittances by \eqref{PabNetworkNeminus4complex} and \eqref{PabNetworkcomplex} are
\begin{equation*}
\P_n(\lambda)=
\begin{dcases}
\dfrac{C \lambda(2n-1)}{ n},\mbox{ if }  \lambda=\pm \dfrac{i}{2\sqrt{CL}},\\
\mbox{not defined, if } \lambda=0,\\
\P_n^{\alpha\beta}(\lambda)=\dfrac{C \lambda\left(\psi_1^{2n-1}+1\right)\left(\psi_1-1\right)}{\left(\psi_1^{2n}-1\right)}, \mbox{otherwise},
\end{dcases}
\end{equation*}
where $\psi_1$ is any root of the equation
\begin{equation}\label{quadraticEqCL}
\psi^2-\left(2+\dfrac{1}{CL\lambda^2}\right)\psi+1=0.
\end{equation}
Firstly, let us point out that 
\begin{equation*}
\lim\limits_{n\rightarrow\infty} \P_n\left(\pm \dfrac{i}{2\sqrt{CL}}\right)=\lim\limits_{n\rightarrow\infty}{\pm \dfrac{C i}{2\sqrt{CL}}} \dfrac{(2n-1)}{ n}=\pm {i}\sqrt{\dfrac{C}{L}}.
\end{equation*}

Let us analyse the sequence $\{\P_n^{\alpha\beta}(\lambda)\}$, $n\rightarrow\infty$.

Since $\psi_1$ is any root of the quadratic equation \eqref{quadraticEqCL}, we can assume, without loss of generality, that $|\psi_1|\le1\le|\psi_2|$. Note that  then $\{\P_n^{\alpha\beta}(\lambda)\}$ has limit if and only if $|\psi_1|<1$. This limit is equal to 
${C \lambda}{(1-\psi_1)}$.
\begin{claim} \label{modulus_psi_1CL}
Let $\lambda\in\Bbb C\setminus\{0\}$, $\lambda^2\ne -\dfrac{1}{4CL}$. Then the condition $|\psi_{1}|=|\psi_{2}|=1$ occurs if and only if $%
\lambda^2 \in \left(-\infty,-\dfrac{1}{4CL}\right)$.
\end{claim}

\begin{proof}
The proof follows the same outline as the proof of Claim \ref{modulus_psi_1}.
\end{proof}

Therefore, for any $\lambda\in\Bbb C\setminus\left(\left(-i\infty,-\dfrac{i}{2\sqrt{CL}}\right]\bigcup\left[\dfrac{i}{2\sqrt{CL}},+i\infty\right)\bigcup\{0\}\right)$ we can write 
\begin{equation}\label{psixiCL}
\psi_1(\lambda)=1+\dfrac{1}{2CL\lambda^2}+\dfrac{1}{2CL\lambda^2}\xi(\gamma(\lambda)),
\end{equation}
where $\xi(\lambda)$ is the square root of ${4\lambda^2CL+1}$, such that $|\psi_1(\lambda)|< 1$. 

Let us denote $\gamma=4CL\lambda^2+1$. Let us consider $\gamma\in \Bbb C\setminus(-i\infty,0)$, i.e. $\gamma=re^{i\phi}$, $\phi\in \left(-\dfrac{\pi}{2},\dfrac{3\pi}{2}\right)$ in polar coordinates (the same cut of the complex plane as in Section \ref{section4}). Then there are two continuous functions $\xi_1, \xi_2$, which give square roots of $\gamma$: 
\begin{equation*}
\xi_1(\gamma)=\sqrt {r} e^{i\dfrac{\phi}{2}},\xi_2(\gamma)=-\xi_1(\gamma)
\end{equation*}
(see Figure \ref{figxi1xi2}).

Since 
\begin{equation*}
\gamma=4\lambda^2CL+1=4 CL \left(\Re \lambda\right)^2-4 CL\left(\Im \lambda\right)^2+1+i 8 CL (\Re\lambda)(\Im\lambda)
\end{equation*}

$\xi_1(\gamma(\lambda))$ and $\xi_2(\gamma(\lambda))$ are defined for all $\lambda\in \Bbb C\setminus \overline{\Lambda}$, where
\begin{equation*}
\overline{\Lambda}=\left\{4 CL \left(\Re \lambda\right)^2-4 CL\left(\Im \lambda\right)^2+1= 0 \mbox{ and } \left(\Re\lambda\right)\left(\Im\lambda\right)<0\right\},
\end{equation*}
see Figure \ref{figxilambdanotdefinedCL}.

\begin{figure}[H]
\centering
\begin{tikzpicture}[scale=0.7]

   \path [draw=none,fill=gray,semitransparent] (-5,-5) rectangle (5,5) ;

    \draw [->] (-5,0) -- (5,0) node [below left]  {$\Re \lambda$};
    \draw [->] (0,-5) -- (0,5) node [below right] {$\Im \lambda$};
   \draw[white, line width = 0.60mm]  plot[smooth,domain=-4.85:0] (\x, {sqrt((\x)^2+4)}); 
    \draw[white, line width = 0.60mm]  plot[smooth,domain=0:4.8] (\x, {-sqrt((\x)^2+4)});
    \draw[dashed, line width = 0.30mm]  plot[smooth,domain=-4.6:0] (\x, {sqrt((\x)^2+4)});
    \draw[dashed, line width = 0.30mm]  plot[smooth,domain=0:4.55] (\x, {-sqrt((\x)^2+4)});

   \node [right,black] at (0,+2) {$\dfrac{i}{2\sqrt{CL}}$};
   \node [left,black] at (0,-2) {$-\dfrac{i}{2\sqrt{CL}}$};

\end{tikzpicture}
\caption{The domain of $\xi_{1,2}(\gamma(\lambda))$}
\label{figxilambdanotdefinedCL}
\end{figure}
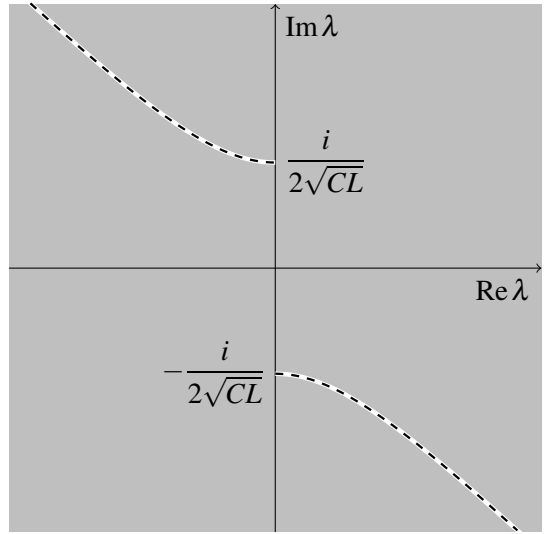

Since the functions
\begin{equation*}
\left|1+\dfrac{1}{2CL\lambda^2}+\dfrac{1}{2CL\lambda^2}\xi_1(\gamma(\lambda))\right| \mbox { and }\left|1+\dfrac{1}{2CL\lambda^2}+\dfrac{1}{2CL\lambda^2}\xi_2(\gamma(\lambda))\right|
\end{equation*}
are continuous on $\lambda$, the choice of the function $\xi_1(\gamma(\lambda))$ or $\xi_2(\gamma(\lambda))$ in \eqref{psixiCL}, can not change inside the domains 
\begin{equation*}
\Omega_1=\left\{\Re^2 \lambda-\Im^2 \lambda>-{\dfrac{1}{4CL}}\right\}\bigcup\left\{\Re^2 \lambda-\Im^2 \lambda<-{\dfrac{1}{4CL}}, (\Im \lambda)(\Re \lambda)>0 \right\},
\end{equation*}
\begin{equation*}
\Omega_2=\left\{\Re^2 \lambda-\Im^2 \lambda<-{\dfrac{1}{4CL}}, \Im \lambda>0, \Re \lambda<0 \right\}
\end{equation*}
and
\begin{equation*}
\Omega_3=\left\{\Re^2 \lambda-\Im^2 \lambda<-{\dfrac{1}{4CL}}, \Im \lambda<0, \Re \lambda>0 \right\}
\end{equation*}
(see Figure \ref{figxilambdanotdefined2CL}).

\begin{figure}[H]
\centering
\begin{tikzpicture}[scale=0.7]

   \path [draw=none,fill=gray,semitransparent]
   (-5,0) -- plot[ smooth,domain=-4.85:0] (\x, {sqrt((\x)^2+4)})-- (0,0) -- cycle ;

   \path [draw=none,fill=gray,semitransparent]
   (5,0) -- plot[ smooth,domain=-4.85:0] (-\x, -{sqrt((\x)^2+4)})-- (0,0) -- cycle ; 

   \path [draw=none,fill=gray,semitransparent] (-5,-5) rectangle (0,0) ;
   \path [draw=none,fill=gray,semitransparent] (5,5) rectangle (0,0) ;


    \draw [->] (-5,0) -- (5,0) node [below left]  {$\Re \lambda$};
    \draw [->] (0,-5) -- (0,5) node [below right] {$\Im \lambda$};
   \draw[white, line width = 0.60mm]  plot[ smooth,domain=-4.9:0] (\x, {sqrt((\x)^2+4)}); 
    \draw[white, line width = 0.60mm]  plot[smooth,domain=0:4.8] (\x, {-sqrt((\x)^2+4)});
    \draw[dashed, line width = 0.30mm]  plot[smooth,domain=-4.6:0] (\x, {sqrt((\x)^2+4)});
    \draw[dashed, line width = 0.30mm]  plot[smooth,domain=0:4.55] (\x, {-sqrt((\x)^2+4)});
    \draw[red, line width = 0.70mm]  (0,2) -- (0,5); 
    \draw[red, line width = 0.70mm]  (0,-2) -- (0,-5); 

   \node [right,black] at (0,+2) {${\dfrac{i}{2\sqrt{CL}}}$};
   \node [left,black] at (0,-2) {$-{\dfrac{i}{2\sqrt{CL}}}$};
   \node [right,gray] at (3,3.5) {$\Omega_1$};
   \node [left,gray] at (-1,3.5) {$\Omega_2$};
   \node [left,gray] at (2,-3.5) {$\Omega_3$};

\end{tikzpicture}
\caption{The domains $\Omega_1$, $\Omega_2$ and $\Omega_3$ for a $CL$-network.}
\label{figxilambdanotdefined2CL}
\end{figure}
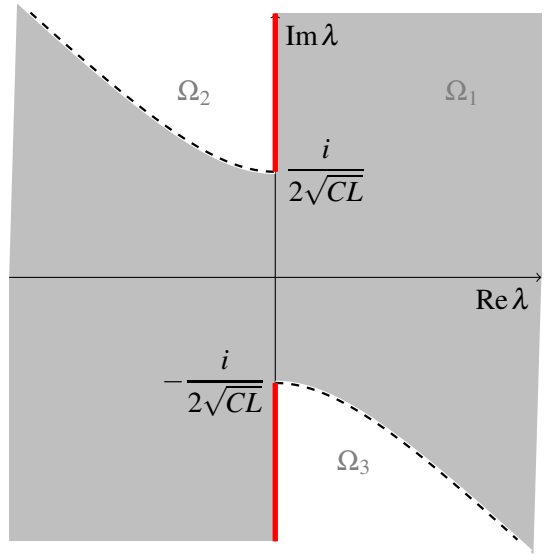

Taking $\lambda=2\in \Omega_1$ we have
\begin{equation*}
\left|1+\dfrac{1}{2CL\lambda^2}+\dfrac{1}{2CL\lambda^2}\xi_1(\gamma(\lambda))\right|
=\left|1+\dfrac{1}{8CL}+\dfrac{1}{8CL}\sqrt{16CL+1}\right|>1.
\end{equation*}
Therefore,
\begin{equation*}
\psi_1(\lambda)=1+\dfrac{1}{2CL\lambda^2}+\dfrac{\xi_2(\gamma(\lambda))}{2CL\lambda^2}, \lambda\in\Omega_1.
\end{equation*}

For the domain $\Omega_2$ in this case we should use a small trick, since taking square root from arbitrary complex number is not easy. This is also a reason, why a practical application of our method for arbitrary ladder network (i.e. with arbitrary $\alpha^{(\lambda)}$ and $\beta^{(\lambda)}$) can be complicated.

Let us take $\gamma=-3-4i$. Then 
\begin{equation*}
\lambda^2=\dfrac{\gamma-1}{4CL}=-\dfrac{1}{CL}-\dfrac{i}{CL}.
\end{equation*}
From the other hand, $\lambda^2=(\Re \lambda)^2-(\Im \lambda)^2+2(\Re \lambda)(\Im \lambda)$. Therefore, we have 
\begin{equation*}
(\Re \lambda)^2-(\Im \lambda)^2=-\dfrac{1}{CL}<-\dfrac{1}{4CL}\mbox{ and }(\Re\lambda) (\Im\lambda)=-\dfrac{1}{CL}<0,
\end{equation*}
i.e. $\lambda\in \Omega_2$ or  $\lambda\in \Omega_3$. Indeed, there exist exactly two complex numbers $\lambda_1, \lambda_2$ such that $\gamma(\lambda_{1,2})=-3-4i$. Without loss of generality we can assume, that $\lambda_1\in \Omega_2$. Then $\lambda_2\in\Omega_3$ and, due to the symmetry $\xi_1(\lambda)=\xi_1(-\lambda)$ and a representation of $\psi_1(\lambda)$, we have $\psi_1(\lambda_1)=\psi_1(\lambda_2)$.  Let us calculate $\psi_1(\lambda_1)$. Using $\lambda_1^2=-\dfrac{1}{CL}-\dfrac{i}{CL}$, we obtain
\begin{equation*}
\begin{split}
\biggl|1\biggr.&-\dfrac{1}{4}+\dfrac{i}{4}+\biggl.\left(-\dfrac{1}{4}+\dfrac{i}{4}\right)\xi_1(-3-4i)\biggr|\\
=&\left|1-\dfrac{1}{4}+\dfrac{i}{4}+\left(-\dfrac{1}{4}+\dfrac{i}{4}\right)(-1+2i)\right|\\
=&\left|\dfrac{1}{2}-i\dfrac{1}{2}\right|<1.
\end{split}
\end{equation*}
Therefore,
\begin{equation*}
\psi_1(\lambda)=1+\dfrac{1}{2CL\lambda^2}+\dfrac{\xi_1(\gamma(\lambda))}{2CL\lambda^2}, \lambda\in\Omega_2\cup\Omega_3.
\end{equation*}

Therefore, we can calculate the effective admittance of an infinite $CL$-ladder for any  $\lambda \in \Bbb C\setminus\left(\overline\Lambda\bigcup\left(-i\infty,-\dfrac{i}{2\sqrt{CL}}\right)\bigcup\left(\dfrac{i}{2\sqrt{CL}},+i\infty\right)\bigcup\{0\}\right)$ as
\begin{equation*}
\P(\lambda)=C\lambda(1-\psi_1)=
\begin{dcases}
-\dfrac{1}{2L\lambda}-\dfrac{\xi_2(\gamma(\lambda))}{2L\lambda}, \lambda\in\Omega_1,\\
-\dfrac{1}{2L\lambda}-\dfrac{\xi_1(\gamma(\lambda))}{2L\lambda}, \lambda\in\Omega_2\cup\Omega_3,\\
\pm {i}\sqrt{\dfrac{C}{L}}, \lambda=\pm \dfrac{i}{2\sqrt{CL}}
\end{dcases}
\end{equation*}
since at the point $\lambda=0$ the admittances of some edges are not defined.

The effective admittance for the $\lambda\in\overline \Lambda$ we can calculate, considering another 
 cut of the plane $\gamma$. Moreover, since the cut $(-i\infty,0)$ has been chosen arbitrary, the limits of the effective admittances from the both sides of the curves $\overline \Lambda$ will coincide and give the required quantity. 

Therefore, we can calculate the effective admittance of the infinite network for any  $\lambda \in \Bbb C\setminus\left(\left(-i\infty,-\dfrac{i}{2\sqrt{CL}}\right)\bigcup\left(\dfrac{i}{2\sqrt{CL}},+i\infty\right)\bigcup\{0\}\right)$.

A natural question is whether the right and left limits of the effective admittance (or of the $\xi_1(\lambda)$, $\xi_2(\lambda)$) at the open intervals $\left(-\infty,-i{\dfrac{1}{2\sqrt{CL}}}\right)$ and $\left(i{\dfrac{1}{2\sqrt{CL}}},+\infty\right)$ coincide (see Figure \ref{figxi1xi2rightleftlimitsCL}).  The answer is negative.

\begin{figure}[H]
\centering
\begin{tikzpicture}[scale=0.7]

   \path [draw=none,fill=gray,semitransparent] (-5,-5) rectangle (5,5);
    \draw [->] (-5,0) -- (5,0) node [below left]  {$\Re \lambda$};
    \draw [->] (0,-5) -- (0,5) node [below right] {$\Im \lambda$};

    \draw[red, line width = 0.70mm]  (0,-5) -- (0,-2); 
    \draw[red, line width = 0.70mm]  (0,5) -- (0,2); 

   \node [right,black] at (0,+2) {$\dfrac{i}{2\sqrt{CL}}$};
   \node [left,black] at (0,-2) {$-\dfrac{i}{2\sqrt{CL}}$};

    \draw [->, blue] (-1,-3) -- (-0.2,-3);
    \draw [->, blue] (-1,-3.5) -- (-0.2,-3.5);
    \draw [->, blue] (-1,-4) -- (-0.2,-4);
    \draw [->, blue] (-1,2.5) -- (-0.2,2.5);
    \draw [->, blue] (-1,3) -- (-0.2,3);
    \draw [->, blue] (-1,3.5) -- (-0.2,3.5);
    \draw [->, blue] (-1,4) -- (-0.2,4);
    \draw [->, blue] (-1,4.5) -- (-0.2,4.5);

    \draw [->, OliveGreen] (1,-2.5) -- (0.2,-2.5);
    \draw [->, OliveGreen] (1,-3) -- (0.2,-3);
    \draw [->, OliveGreen] (1,-3.5) -- (0.2,-3.5);
    \draw [->, OliveGreen] (1,-4) -- (0.2,-4);
    \draw [->, OliveGreen] (1,-4.5) -- (0.2,-4.5);
    \draw [->, OliveGreen] (1,3) -- (0.2,3);
    \draw [->, OliveGreen] (1,3.5) -- (0.2,3.5);
    \draw [->, OliveGreen] (1,4) -- (0.2,4);

\end{tikzpicture}
\caption{The limits of $\xi_1(\lambda)$ and $\xi_2(\lambda)$ for a $CL$-network}
\label{figxi1xi2rightleftlimitsCL}
\end{figure}
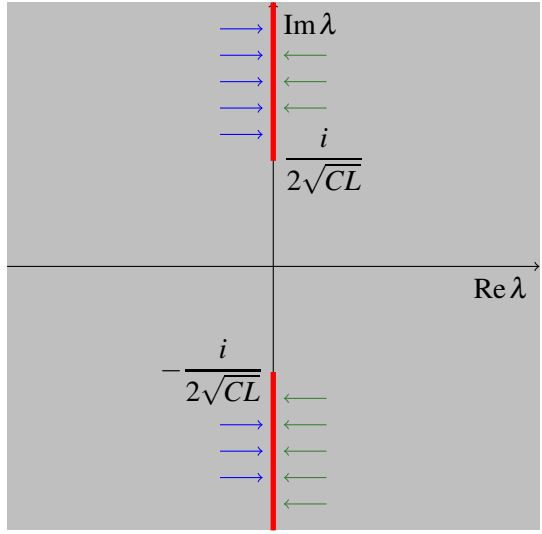

We will consider upper half case $\left(i{\dfrac{1}{2\sqrt{CL}}},+i\infty\right)$. The other case can be calculated similarly.

Let $\lambda=\epsilon+ i\omega$, $1\gg\epsilon>0$, $\omega\in\left({\dfrac{1}{2\sqrt{CL}}},+\infty\right)$, i.e. $\lambda\in\Omega_1$. Then
\begin{equation*}
\gamma(\epsilon+i\omega)=4 CL (\epsilon^2-\omega^2)+1+i 8 CL \epsilon\omega
\end{equation*}
and 
\begin{equation}\label{limitEffadmCL1}
\begin{split}
\lim_{\lambda\rightarrow+i\omega}\P(\lambda)=&\lim_{\epsilon\rightarrow 0}\left(-\dfrac{1}{2L(\epsilon+i\omega)}-\dfrac{\xi_2(\gamma(\epsilon+i\omega))}{2L(\epsilon+i\omega)}\right)\\
=&  -\dfrac{1}{2Li\omega}+\dfrac{i}{2Li\omega}\sqrt{4CL{\omega^2}-1}= \dfrac{i}{2L\omega}+\sqrt{\dfrac{C}{L}-\dfrac{1}{4L^2{\omega^2}}},
\end{split}
\end{equation}
since $\Re \gamma<0$, $|\Im \gamma|\ll 1$ provides $\Im \xi_2<0$. 

Let $\lambda=-\epsilon+ i\omega$, $1\gg\epsilon>0$, $\omega\in\left({\dfrac{1}{2\sqrt{CL}}},+\infty\right)$, i.e. $\lambda\in\Omega_2$. Then
\begin{equation*}
\gamma(-\epsilon+i\omega)=4 CL (\epsilon^2-\omega^2)+1-i 8 CL \epsilon\omega
\end{equation*}
and 
\begin{equation*}
\begin{split}
\lim_{\lambda\rightarrow-i\omega}\P(\lambda)=&\lim_{\epsilon\rightarrow 0}\left(-\dfrac{1}{2L(-\epsilon+i\omega)}-\dfrac{\xi_1(\gamma(-\epsilon+i\omega))}{2L(-\epsilon+i\omega)}\right)\\
=&  -\dfrac{1}{2Li\omega}-\dfrac{i}{2Li\omega}\sqrt{4CL{\omega^2}-1}= \dfrac{i}{2L\omega}-\sqrt{\dfrac{C}{L}-\dfrac{1}{4L^2{\omega^2}}},
\end{split}
\end{equation*}
since $\Re \gamma<0$, $|\Im \gamma|\ll 1$ provides $\Im \xi_1>0$. 

The limit \eqref{limitEffadmCL1} coincides with the one, stated by  S.H. Yoon in \cite{Yoon} . Indeed, by \cite[p. 286]{Yoon} we have the effective admittance
\begin{align*}
\P=&\dfrac{1}{Z}=\dfrac{1}{{-i}/{(2\omega C)}+\sqrt{L/C-1/(2\omega C)^2}}=\dfrac{{-i}/{(2\omega C)}-\sqrt{L/C-1/(2\omega C)^2}}{({-i}/{(2\omega C)})^2-(L/C-1/(2\omega C)^2)}\\
=&-\dfrac{C}{L}\left(\dfrac{-i}{2\omega C}-\sqrt{\dfrac{L}{C}-\dfrac{1}{(2\omega C)^2}}\right)=\dfrac{i}{2L\omega }+\sqrt{\dfrac{C}{L}-\dfrac{1}{4L^2\omega^2}}.
\end{align*}

Note, that the above described method for calculation of the effective admittances for $LC-$ and $CL-$ networks can be used for any infinite ladder network with the admittances $\alpha^{(\lambda)}, \beta{(\lambda)}$, although the choice of the one of two solutions of the quadratic equation
\begin{equation}\label{psicomplex}
\psi^2-\left(2+\mu\right)\psi+1=0.
\end{equation}  
in different domains of the complex plane $\lambda$ can be quite complicated to do practically for a given infinite ladder network.
Therefore, we have presented a general method for calculation of the effective impedance of infinite ladder network and illustrated it with two physically important infinite ladder networks.

\section*{Acknowledgement}
The author thanks her scientific advisor, Professor Alexander Grigor'yan, for helüful and fruitful discussions on the topic.


\begin{thebibliography}{99}

\bibitem{AlonsoRuiz} Patricia Alonso Ruiz. \emph{ Power dissipation in fractal {Feynman}-{Sierpinski} {A}{C} circuits }. Journal of Mathematical Physics 2017. Vol. 58, 215--237.  {\url{http://dx.doi.org/10.1063/1.4994197}}

\bibitem{Barlow} Martin T. Barlow. \emph{Random Walks and Heat Kernels on Graphs}. London Mathematical Society, Lecture Note Series: 438. Cambridge University Press, 2017.  {\url{http://dx.doi.org/10.1017/9781107415690}}
 

\bibitem{Brune} O. Brune. \emph{Synthesis of a finite two-terminal network whose driving-point impedance is a prescribed function of frequency. Thesis (Sc. D.)}. Massachusetts Institute of Technology, Dept. of Electrical Engineering, Massachusetts, 1931. {\url{http://dx.doi.org/10.1002/sapm1931101191}}

\bibitem{ChenTeplyaev} Joe P.  Chen, Luke G. Rogers, Loren Anderson, Ulysses Andrews, Antoni Brzoska, Aubrey Coffey, Hannah Davis, Lee Fisher, Madeline Hansalik, Stephew Loew, Alexander Teplyaev. \emph{Power dissipation in fractal {A}{C} circuits}. Journal of Physics A: Mathematical and Theoretical, 2017. Vol. 50, n. 32. {\url{http://dx.doi.org/10.1088/1751-8121/aa7a66}}


\bibitem{DS} P.G. Doyle, J.L. Snell. \emph{Random walks and electric networks}. Carus Mathematical Monographs 22,  Mathematical Association of America. Washington, DC, 1984. {\url{http://dx.doi.org/10.5948/UPO9781614440222}}

\bibitem{Enk} S. J. van Enk. \emph{Paradoxical   behavior   of   an   infinite   ladder   network   of   inductors   and   capacitors}. American journal of physics, 2000. Vol. 68, n. 9, 854--856.


\bibitem{Feynman1} Richard P. Feynman, Robert B. Leighton, Matthew Sands.
\emph{The Feynman lectures on physics, Volume 1: Mainly mechanics, radiation, and heat}. Addison-Wesley publishing company. Reading, Massachusetts, Fourth printing -- 1966. {\url{http://dx.doi.org/10.1063/1.3051743}}

\bibitem{Feynman2} Richard P. Feynman, Robert B. Leighton, Matthew Sands.
\emph{The Feynman lectures on physics, Volume 2: Mainly electromagnetism and matter}. Addison-Wesley publishing company. Reading, Massachusetts, Fourth printing -- 1966. {\url{http://dx.doi.org/10.1063/1.3051743}}


\bibitem{Grimmett} G.  Grimmett. \emph{Probability on Graphs: Random Processes on Graphs and Lattices}. Cambridge University Press. New York, 2010.  {\url{http://dx.doi.org/10.1017/CBO9780511762550}}


\bibitem{Hughes} Edward Hughes. \emph{Electrical and electronic technology}. Pearson Education Limited, Tenth edition -- 2008.

\bibitem{Klimo} Paul Klimo. \emph{On the impedance of infinite {L}{C} ladder networks}. European journal of physics, 2016. Vol. 38, n.1, 1--9. {\url{http://dx.doi.org/10.1088/0143-0807/38/1/015805}}

\bibitem{LPW}David A. Levin, Yuval Peres, Elizabeth L. Wilmer.
        \emph{Markov Chains and Mixing Times}. AMS University Lecture Series. Providence, Rhode Island, 2009. {\url{http://dx.doi.org/10.1090/mbk/058}}

\bibitem{Muranova3} Anna Muranova.  \emph{On the effective impedance of finite and infinite networks}. arXiv e-prints, page arXiv:1908.10025, August 2019.

\bibitem{Muranova1} Anna Muranova.  \emph{On the notion of effective impedance}. arXiv e-prints, page arXiv:1905.02047, May 2019.




\bibitem{Soardi} Paolo M. Soardi. \emph{Potential Theory on Infinite Networks}. Springer-Verlag, Berlin Heidelberg, 1994. {\url{http://dx.doi.org/10.1007/BFb0073995}}

\bibitem{UA} C. Ucak, C. Acar. \emph{Convergence and periodic solutions for the input impedance of a standard ladder network}. European journal of physics, 2007. Vol. 28, n. 2, 321--329. {\url{http://dx.doi.org/10.1088/0143-0807/28/2/017}}

\bibitem{UY} C. Ucak, K. Yegin. \emph{Understanding the behaviour of infinite ladder circuits}. European journal of physics, 2008. Vol. 29, n. 6, 1201--1209. {\url{http://dx.doi.org/10.1088/0143-0807/29/6/009}}


\bibitem{Woess} Wolfgang Woess. \emph{Random Walks on Infinite Graphs and Groups}. Cambridge Tracts in Mathematics: 138. Cambridge University Press, 2000. {\url{http://dx.doi.org/10.1017/CBO9780511470967}}

\bibitem{Yoon} Sung Hyun Yoon. \emph {Ladder-type circuits revisited}. {European journal of physics}, 2007. Vol. 22, n. 22, 277--288. {\url{http://dx.doi.org/10.1088/0143-0807/28/2/013}}



\end{thebibliography}
\end{document}